\documentclass[12pt]{article}
\author{Marc Munsch\footnote{Supported by the Austrian Science Fund (FWF), START-project Y-901 ``Probabilistic methods in analysis and number theory" headed by Christoph Aistleitner.}\\ TU Graz, Austria \\ munsch@math.tugraz.at \\
and
\\ Tim Trudgian\footnote{Supported by ARC Future Fellowship FT160100094.}  \\ ANU, Australia \\ timothy.trudgian@anu.edu.au}
\title{Square-full primitive roots}
\usepackage{indentfirst}
\usepackage{url}
\usepackage{enumerate}
\usepackage{amsthm}
\usepackage{amsmath}
\usepackage{comment}
\usepackage{fullpage}
\usepackage{amssymb}
\usepackage{booktabs}

\newtheorem{thm}{Theorem}
\newtheorem{prop}{Proposition}

\newtheorem{Lem}{Lemma}
\newtheorem{remark}{Remark}
\newtheorem{cor}{Corollary}

\begin{document}
\maketitle
\begin{abstract}
\noindent
We use character sum estimates to give a bound on the least square-full primitive root modulo a prime. Specifically, we show that there is a square-full primitive root mod $p$ less than $p^{2/3 + 3/(4 \sqrt{e})+ \epsilon}$, and we give some conditional  bounds.
\end{abstract}
\textit{MSC Codes:} 11A07, 11L40.

\section{Introduction}
The distribution of primitive roots is a rich topic in analytic number theory. Let $g(p)$ denote the least positive primitive root modulo a prime $p$; the classic breakthrough result of Burgess \cite{Burgess} is that $g(p)\ll p^{1/4 + \epsilon}$. This seems impossible to improve unconditionally without a completely novel approach. In the absence of such an improvement many authors have proved the existence of small primitive roots with additional properties. For example Ha \cite{MR3084293} has shown that for all primes $p$ there exists a  prime primitive root $\ll p^{3.1}$; Cohen and Trudgian \cite{Tim} showed that for all primes $p$ there is a square-free primitive root less than $p^{0.96}$ --- this was improved by Hunter \cite{Hunter} to $p^{0.88}$. In this article we consider square-full primitive roots.

Recall that an integer $n$ is square-full if for all primes $p|n$ we have $p^{2}|n$. In particular, all squares are square-full. Since squares cannot be primitive roots modulo $p$ (for $p>2$) questions about square-full primitive roots rely on peculiarities in the distribution of square-full-but-not-square integers. 


Shapiro \cite{MR693458} was the first to investigate square-full primitive roots. Let $N_{\blacksquare}(p, x)$ denote the number of square-full primitive roots modulo $p$ that do not exceed $x$. In \cite[Thm.\ 8.5.A.1]{MR693458} Shapiro proves
\begin{equation}\label{ballot}
N_{\blacksquare}(p, x) = \frac{\phi(p-1)}{p-1} \left\{ c \sqrt{x} + O( x^{1/3} p^{1/6} (\log p)^{1/3} 2^{\omega(p-1)})\right\},
\end{equation}
where $\phi(n)$ is Euler's function, $\omega(n)$ denotes the number of distinct prime divisors of $n$, and where 
\begin{equation*}
c = 2\left( 1 - \frac{1}{p}\right) \sum_{(q|p) = -1} \frac{\mu^{2}(q)}{q^{3/2}},
\end{equation*}
in which the sum is over those $q$ that are quadratic non-residues modulo $p$. If $c$ were a constant one could conclude, as Shapiro does on p.\ 308 that the least positive square-full primitive root mod $p$ is at most $p^{1 + \epsilon}$. However, since there can be long runs of quadratic residues, $c$ could be very small compared with $p$. In this article we improve on the error terms in (\ref{ballot}) and estimate $c$ correctly to prove
\begin{thm}\label{lower}
 We have $$ N_{\blacksquare}(p, x) =\frac{\phi(p-1)}{p-1} \left\{\frac{1}{\zeta(3)} \left(1+\frac{1}{p}+\frac{1}{p^2}\right)^{-1} C_p x^{1/2} + O(x^{1/3} (\log x) p^{1/9}(\log p)^{1/6}2^{\omega(p-1)})\right\}, $$ where $C_p \gg p^{-\frac{1}{8\sqrt{e}}}$.
\end{thm}
Denote by $g_{\blacksquare}(p)$ the least square-full primitive root modulo $p$. Thus, directly from Theorem \ref{lower} we deduce
\begin{cor}\label{leastfull} $g_{\blacksquare}(p) \ll p^{2/3+3/4\sqrt{e}} (\log p)^{7} 2^{6\omega(p-1)}.$   \end{cor}
The exponent of $p$ in Corollary \ref{leastfull} is $1.121\ldots$. It would be interesting to see whether these methods, all of which can be made explicit, could be improved to bring this exponent below unity. Such an improvement could test the hypothesis made by Cohen and Trudgian \cite[p.\ 16]{Tim}, that $1,052,041$ is the largest prime $p$ with $g_{\blacksquare}(p)\geq p$. 

The outline of this paper is as follows. In \S \ref{burger} we use character sum estimates to prove Theorem \ref{lower}. In \S \ref{ketchup} we obtain better bounds on $g_{\blacksquare}(p)$ conditional on the Generalized Riemann Hypothesis. Finally, in \S \ref{fries} we adapt the methods of this paper to recover a known result about square-free primitive roots.

\subsection*{Acknowledgements}
We are grateful to Igor Shparlinski who suggested the problem to us.

\section{Character sums and average results}\label{burger}
We have (see Shapiro, p.\ 307)
\begin{equation}\label{cheese}
N_{\blacksquare}(p, x) = \frac{\phi(p-1)}{p-1} \left\{  \sum_{\substack{m\leq x \\ m \;\textrm{square-full}}} \chi_0(m) - \sum_{\substack{m\leq x \\ m \;\textrm{square-full}\\ d=2}} \chi(m) + \sum_{\substack{d|(p-1)\\ d>2}} \frac{\mu(d)}{\phi(d)} \sum_{\chi \in \Gamma_{d}} \sum_{\substack{m\leq x \\ m\; \textrm{square-full}}} \chi(m)    \right\},
\end{equation}
where $\chi_0$ denotes the principal character modulo $p$ and $\Gamma_{d}$ the set of characters of order $d$. Therefore the first sum in (\ref{cheese}) counts the number of square-full integers up to $x$ that are co-prime to $p$; in the second sum (over $d=2$) we consider only quadratic characters.

We will make use of the famous character estimate of Burgess (see \cite{Burgess}).

\begin{Lem}\label{Burgess}\upshape[{Burgess inequality}] \itshape
For every integer $r\geq 2$ and every non-principal character $\chi$ modulo $p$, we have 
\begin{equation*}
 \sum_{m \leq x} \chi(m) \ll  x^{1- 1/r} p^{\frac{r+1}{4r^{2}}} (\log p)^{\frac{1}{2r}},
\end{equation*} where the implied constant only depends on $r$.
\end{Lem} The next result, giving the contribution of the principal and the quadratic characters in (\ref{cheese}), is essentially contained in Lemma $1.3$ of \cite{Marc}.

\begin{Lem}\label{quadratic}
 Let $\chi$ be a Dirichlet character modulo $p$, $L(s,\chi)$ be the associated Dirichlet $L$- function and $\epsilon$ be a fixed positive number. Then we have
$$\displaystyle{\sum_{\substack{m\leq x \\ m\; \textrm{square-full}}} \chi(m)}=\begin{cases}
\displaystyle{\frac{L\left(3/2,\chi_0\right)}{\zeta(3)}\left(1+\frac{1}{p}+\frac{1}{p^2}\right)^{-1}x^{1/2}+  \frac{\zeta\left(2/3\right)}{\zeta(2)}\left(\frac{1-p^{-2/3}}{1+p^{-1}}\right)x^{1/3}}\\ \vspace{4mm}
\displaystyle{+O(x^{1/6+\epsilon}p^{\epsilon})} \mbox{ if $\chi=\chi_0$}, \\
\displaystyle{\frac{L\left(3/2,\chi_2\right)}{\zeta(3)}\left(1+\frac{1}{p}+\frac{1}{p^2}\right)^{-1}x^{1/2}}\\ \vspace{4mm}
\displaystyle{+ O\left(x^{1/4}(\log x)^{1/2}p^{3/32+\epsilon}\right)} \mbox{ if $\chi$ is a quadratic character}. 
\end{cases}
$$
\end{Lem} These terms will in fact give the main contribution in (\ref{cheese}). We bound the contribution of higher order characters in the following result.

\begin{Lem}\label{highercharac}
Suppose that $\chi^2 \neq \chi_0$ is a Dirichlet character modulo $p$. Then, we have

\begin{equation*}
\sum_{\substack{m\leq x \\ m\; \textrm{square-full}}} \chi(m) \ll x^{1/3}p^{1/9}(\log p)^{1/6}(\log x).  \end{equation*}

\end{Lem}

\begin{proof}Each square-full integer $m$ can be written in a unique way $m = a^2b^3$ with $b$ square-free. Thus we have
\begin{equation*}
 \sum_{\substack{m\leq x \\ m\; \textrm{square-full}}} \chi(m) = \sum_{\substack{a^{2} b^{3} \leq x\\ b \; \textrm{square-free}}} \chi(a^{2} b^{3})  =  \sum_{b\leq x^{1/3}} \mu^{2}(b) \chi^{3}(b) \sum_{a \leq (\frac{x}{b^{3}})^{1/2}} \chi^{2}(a)
 \end{equation*} We now apply the Burgess inequality (Lemma \ref{Burgess}) with $r=3$ to the final inner sum --- this is fine since $\chi^{2} \neq \chi_{0}$. We have 
 
 $$\sum_{\substack{m\leq x \\ m\; \textrm{square-full}}} \chi(m) \ll x^{1/3}p^{1/9}(\log p)^{1/6} \sum_{b\leq x^{1/3}}\frac{1}{b} \ll x^{1/3}p^{1/9}(\log p)^{1/6} (\log x), $$ which concludes the proof.

\end{proof}


\subsection{Proof of Theorem \ref{lower}} 
For the final sum in (\ref{cheese}) we bound the inner sum using Lemma \ref{highercharac}, then, using the triangle inequality and the fact that $\Gamma_{d}$ denotes the  $\phi(d)$ characters of order $d$ we obtain
\begin{equation*}
\left|\sum_{\substack{d|(p-1)\\ d>2}} \frac{\mu(d)}{\phi(d)} \sum_{\chi \in \Gamma_{d}} \sum_{\substack{m\leq x \\ m\; \textrm{square-full}}} \chi(m)\right| \ll 2^{\omega(p-1)}x^{1/3}p^{1/9}(\log p)^{1/6}(\log x).
\end{equation*}
Regrouping the terms of size $x^{1/2}$, using Lemma \ref{quadratic}, and noting that the error terms appearing here are smaller than the ones encountered above, we end up with the expected asymptotic formula in Theorem \ref{lower} where 

$$C_p = \sum_{n\geq 1} \frac{\chi_0(n)-\chi_2(n)}{n^{3/2}} = 2\sum_{\substack{n\geq 1 \\ \chi_2(n)=-1}}\frac{1}{n^{3/2}},$$
in which $\chi_{2}$ denotes the quadratic character modulo $p$.

As pointed out in the final section of \cite{Banks} (precisely Theorem $4.1$), using Hildebrand's extension of the Burgess inequality (see \cite{Hildebrand}) and the work of Granville and Soundararajan \cite{GranvilleSound} on completely multiplicative functions, there exists a positive proportion (greater than 34 \%) of integers less than $p^{\frac{1}{4\sqrt{e}}}$ that are quadratic non-residues modulo $p$. Thus, 

$$C_p \gg \sum_{\substack{p^{\frac{1}{4\sqrt{e}}} \leq n\leq 2p^{\frac{1}{4\sqrt{e}}}\\ (n|p) = -1}} \frac{1}{n^{3/2}} \gg p^{-\frac{1}{8\sqrt{e}}},$$ which concludes the proof.

  \section{Conditional results}\label{ketchup}
We assume the Generalized Riemann Hypothesis (GRH) throughout this section in order to pursue sharper bounds on $g_{\blacksquare}(p)$. We note that the Burgess bound $g(p) \ll p^{1/4 +\epsilon}$ has been improved under GRH by many authors, the strongest result to date being by Shoup \cite{MR1106981}, who showed that $g(p) \ll (\log p)^{6+ \epsilon}$. 
 
We shall focus on a subset of the square-full numbers, namely those numbers of the type $n=q^2r^3$ where $q$ and $r$ are primes. We denote this set by $\mathcal{S}$ and let $N_{\mathcal{S}}(x)$ be the number of elements in $\mathcal{S}$ which are primitive roots modulo $p$ and which do not exceed $x$. We prove

 \begin{thm}\label{thmGRH} 
 Assuming the Generalized Riemann Hypothesis we have $$ N_{\mathcal{S}}(x)(p, x) =\frac{\phi(p-1)}{p-1} \left\{\sum_{\substack{r\leq x^{1/3} \\ r \; \textrm{prime}}}(\chi_0(r)-\chi_2(r)) Li\left(\frac{x^{1/2}}{r^{3/2}}\right)+ O(2^{\omega(p-1)}x^{1/3}\log^2 (px))\right\}, $$ where $\textrm{Li}(x)$ is the logarithmic integral.
\end{thm}
To prove Theorem \ref{thmGRH} we require some preparatory lemmas. We first estimate the contribution of the principal and quadratic characters in Theorem \ref{thmGRH}.

 \begin{Lem}\label{quadraticGRH}
 Let $\chi$ be a Dirichlet character modulo $p$ such that $\chi^2 = \chi_0$. Then assuming GRH we have
$$\displaystyle{\sum_{\substack{m\in \mathcal{S} \\ m\leq x}} \chi(m)}=
\displaystyle{\sum_{\substack{r\leq x^{1/3}\\ r \; \textrm{prime}}}  \chi(r) Li\left(\frac{x^{1/2}}{r^{3/2}}\right)} + \displaystyle{O\left(x^{1/3}\log x\right)}.\\
$$

\end{Lem}

 \begin{proof}  
 We have
\begin{equation*}
 \sum_{\substack{m\in \mathcal{S} \\ m\leq x}} \chi(m) = \sum_{\substack{q^{2} r^{3} \leq x\\ q,r \; \textrm{primes}}} \chi(q^{2} r^{3})  =  \sum_{\substack{r\leq x^{1/3}\\ r \; \textrm{prime}}}  \chi(r) \sum_{q \leq (\frac{x}{r^{3}})^{1/2}} \chi_0(q).
 \end{equation*} We now apply the GRH version of the Prime Number Theorem (see \cite[Chapter $18$]{Davenport}) to get
 
 $$\sum_{\substack{r\leq x^{1/3} \\ r \; \textrm{prime}}} \chi(r) \textrm{Li}\left(\frac{x^{1/2}}{r^{3/2}}\right) + O\left(x^{1/4}\log x\sum_{r\leq x^{1/3}}\frac{1}{r^{3/4}} \right) = \sum_{\substack{r\leq x^{1/3} \\ r \; \textrm{prime}}}  \chi(r) \textrm{Li}\left(\frac{x^{1/2}}{r^{3/2}}\right) + O\left(x^{1/3}\log x\right).$$

 \end{proof}

  We will use the following result (see, e.g.\ \cite[Chapter $20$]{Davenport}) in order to bound the contribution of the higher order characters.
 \begin{Lem}\label{GRHbound} Suppose that $\chi$ is a non-principal character modulo $p$. Then, under GRH we have the following bound 
 
$$ \sum_{\substack{q\leq x \\ q \; \textrm{prime}}} \chi(q) \ll \sqrt{x}\log^2 (px).  $$ 
   \end{Lem}
   The last lemma we need is a conditional analogue of Lemma \ref{highercharac}.

\begin{Lem}\label{highercharacgrh}
Suppose that $\chi^2 \neq \chi_0$ is a Dirichlet character modulo $p$. Then, under GRH we have the following bound 

\begin{equation*}
\sum_{\substack{m\in \mathcal{S} \\ m\leq x}} \chi(m) \ll x^{1/3}\log^2 (px).  \end{equation*}

\end{Lem}

\begin{proof}
We have
\begin{equation*}
 \sum_{\substack{m\in \mathcal{S} \\ m\leq x}} \chi(m) = \sum_{\substack{q^{2} r^{3} \leq x\\ q,r \; \textrm{primes}}} \chi(q^{2} r^{3})  =  \sum_{\substack{r\leq x^{1/3} \\ r \; \textrm{prime}}}  \chi^{3}(r) \sum_{q \leq (\frac{x}{r^{3}})^{1/2}} \chi^{2}(q)
 \end{equation*} We now apply Lemma \ref{GRHbound} to the final inner sum and we get
 
 
 $$\sum_{\substack{m\in \mathcal{S} \\ m\leq x}} \chi(m) \ll x^{1/4}\log^2 (px) \sum_{r\leq x^{1/3}}\frac{1}{r^{3/4}} \ll x^{1/3}\log^2 (px)$$ which concludes the proof.

\end{proof}

 \subsection{Proof of Theorem \ref{thmGRH}}
 
Analogously to (\ref{cheese}), we have 
\begin{equation*}\label{numberS}
 N_{\mathcal{S}}(x)(p, x) = \frac{\phi(p-1)}{p-1} \left\{  \sum_{\substack{m\in \mathcal{S} \\ m\leq x}} \chi_0(m) -  \sum_{\substack{m\in \mathcal{S} \\ m\leq x}} \chi_2(m) + \sum_{\substack{d|(p-1)\\ d>2}} \frac{\mu(d)}{\phi(d)} \sum_{\chi \in \Gamma_{d}} \sum_{\substack{m\in \mathcal{S} \\ m\leq x}}  \chi(m)    \right\}.
\end{equation*} As in the proof of Theorem \ref{lower}, we have, using Lemma \ref{highercharacgrh},

\begin{equation*}
\sum_{\substack{d|(p-1)\\ d>2}} \frac{\mu(d)}{\phi(d)} \sum_{\chi \in \Gamma_{d}} \sum_{\substack{m\in \mathcal{S} \\ m\leq x}}  \chi(m)\ll 2^{\omega(p-1)}\log^2 (px).
\end{equation*} Grouping the terms of size $x^{1/2}$, using Lemma \ref{quadraticGRH}, and remarking that the error terms are smaller than the ones encountered above concludes the proof. 

\vspace{2mm}

We now deduce the following bound, which is a conditional improvement to Corollary~\ref{leastfull}.
 \begin{cor} Assuming GRH, we have $g_{\blacksquare}(p) \ll 2^{6\omega(p-1)}\omega(p-1)(\log p)^{18} \log\log p.$   \end{cor}

 \begin{proof} We have $$\sum_{\substack{r\leq x^{1/3} \\ r \; \textrm{prime}}}(\chi_0(r)-\chi_2(r)) \textrm{Li}\left(\frac{x^{1/2}}{r^{3/2}}\right) = 2\sum_{\substack{r\leq x^{1/3} \\ r \; \textrm{prime and } (r|p) = -1}}\textrm{Li}\left(\frac{x^{1/2}}{r^{3/2}}\right).$$ As a consequence of the work of Ankeny \cite{Ankeny}, we deduce the existence of a positive proportion of primes $r$ less than $C\log^2 p$ which are quadratic non-residues modulo $p$, for some explicit constant $C>0$. Thus, 

$$\sum_{\substack{r\leq x^{1/3} \\ r \; \textrm{prime and } (r|p) = -1}} \textrm{Li}\left(\frac{x^{1/2}}{r^{3/2}}\right) \gg \sum_{\substack{C\log^2 p \leq r\leq 2C \log^2 p \\ (r|p) = -1}} \textrm{Li}\left(\frac{x^{1/2}}{r^{3/2}}\right)  \gg \frac{x^{1/2}}{\log x \log p}.$$ Comparing this with the error term of Theorem \ref{thmGRH} completes the proof.\end{proof}

\begin{remark}We could possibly obtain a better exponent of the logarithm using a smooth counting for the primes (like in \cite{Ankeny}) which would lead to better bounds for Lemma \ref{GRHbound}.  \end{remark}

 \begin{remark} We could have obtained a similar result over all square-full numbers. However, in that case, assuming GRH, the Lindel\"{o}f Hypothesis gives the bound $\displaystyle{\sum_{m\leq x} \chi(m) \ll x^{1/2+\epsilon}}$ which does not allow us to obtain a logarithmic error term. 
 \end{remark}

 \section{Least square-free primitive root}\label{fries}
Although Zhai and Liu \cite{squarefree} have already proved the existence of small square-free primitive roots, we use this short section to reprove this fact. 
 
\begin{Lem}\label{sqrfree} 
 Let $\chi$ be a non-principal character modulo $p$. Then for all integers $r\geq 3$, we have
$$\sum_{m\leq x}\mu^2(m)\chi(m)\ll 
x^{1-1/r}p^{\frac{r+1}{4r^2}}(\log p)^{\frac{1}{2r}}. $$

\end{Lem}

\begin{proof}

Since $\mu(m)^2=\sum_{d^2\mid m}\mu(d)$, we obtain

$$\sum_{m\leq x}\mu^2(m)\chi(m)=\sum_{d\leq x^{1/2}} \mu(d)\chi(d^2)\sum_{m\leq x/d^2}\chi(m).$$ Then we apply the Burgess inequality (Lemma \ref{Burgess}) to show that

$$ \sum_{m\leq x/d^2}\chi(m) \ll \left(\frac{x}{d^2}\right)^{1/2}p^{\frac{r+1}{4r^2}}(\log p)^{\frac{1}{2r}}.$$ After summing over $d$ we obtain the desired result. 
\end{proof}

 In comparison with the previous section, we easily derive
 
 \begin{prop}
Denote by $g^{\square}(p)$ the least square-free primitive root modulo $p$. Then, for every $\epsilon>0$, we have $g^{\square}(p) \ll p^{1/4+\epsilon}$.
\end{prop}
\begin{proof}Denote by $N^{\square}(p,x)$ the number of square-free primitive roots modulo $p$ that do not exceed $x$. Proceeding similarly as in the proof of Theorem $1$ and Corollary \ref{leastfull}, and applying Lemma \ref{sqrfree}, we end 
up with 

$$N^{\square}(p,x)= \frac{p\phi(p-1)}{p^2-1} \frac{6}{\pi^2} x  + O(x^{1-1/r}p^{\frac{r+1}{4r^2}}(\log p)^{\frac{1}{2r}}).$$ For every integer $r$, the error terms are dominated by the main term as long as $x\gg p^{1/4+1/4r+\epsilon}$. Sending $r$  to infinity completes the proof.

\end{proof}

\end{document}